\documentclass[envcountsame, envcountsect]{llncs}

\usepackage{amssymb,amsmath,url,mathrsfs,graphicx,wasysym} 

\newcommand{\Z}{{\mathbf Z}}
\newcommand{\F}{{\mathbf F}}
\renewcommand{\O}{{\mathcal O}}
\newcommand{\Q}{\mathbf{Q}}
\newcommand{\RR}{\mathbf{R}}
\newcommand{\C}{\mathbf{C}}
\newcommand{\Qbar}{{\overline{\Q}}}

\newcommand{\Fbar}{{\overline{\F}}}
\newcommand{\p}{{\mathfrak p}}
\newcommand{\q}{{\mathfrak q}}
\renewcommand{\r}{{\mathfrak r}}
\newcommand{\R}{{\mathfrak R}}
\newcommand{\pibar}{\overline{\pi}}
\newcommand{\abs}[1]{\lvert #1 \rvert}

\newcommand{\ghat}{{\widehat{g}}}

\DeclareMathOperator{\Jac}{Jac}
\DeclareMathOperator{\Char}{char}

\newcommand{\End}{\operatorname{End}}
\newcommand{\Gal}{\operatorname{Gal}}

\DeclareMathOperator{\N}{N}
\DeclareMathOperator{\Img}{Im}

\spnewtheorem{open}[theorem]{Open problem}{\bfseries}{\itshape}
\spnewtheorem{algorithm}[theorem]{Algorithm}{\bfseries}{\rmfamily}
\spnewtheorem{ex}[theorem]{Example}{\bfseries}{\rmfamily}
\spnewtheorem{rem}[theorem]{Remark}{\bfseries}{\rmfamily}

\begin{document}

\title{Abelian varieties with prescribed \\ embedding degree}
\author{David Freeman\inst{1}, Peter Stevenhagen\inst{2},
           and Marco Streng\inst{2}}

\institute{University of California, Berkeley\footnote{The
first author is supported by a National Defense Science and
Engineering Graduate Fellowship} \\
\email{dfreeman@math.berkeley.edu}
\and Mathematisch Instituut, Universiteit Leiden\\
\email{psh,streng@math.leidenuniv.nl}}

\maketitle

\begin{abstract}
We present an algorithm that, on input of a CM-field $K$,
an integer $k\ge1$, and a prime $r\equiv1\bmod k$,
constructs a $q$-Weil number $\pi\in\O_K$ corresponding to
an ordinary, simple abelian variety $A$ over the
field $\F$ of $q$ elements
that has an $\F$-rational point of order $r$
and embedding degree~$k$ with respect to $r$.
We then discuss how CM-methods over $K$ can be used to
explicitly construct $A$.
\end{abstract}



\section{Introduction}
\label{s:intro}

Let $A$ be an abelian variety defined over a finite field $\F$,
and $r\ne \Char(\F)$ a prime number dividing the order of the group $A(\F)$.
Then the \emph{embedding degree} of $A$ with respect to $r$ is the
degree of the field extension $\F\subset \F(\zeta_r)$ obtained by adjoining
a primitive $r$-th root of unity $\zeta_r$ to $\F$.

The embedding degree is a natural notion in pairing-based
cryptography, where $A$ is taken to be the Jacobian of a
curve defined over $\F$.
In this case, $A$ is principally polarized and we have
the non-degenerate \emph{Weil pairing}
\[
e_r: A[r]\times A[r] \longrightarrow \mu_r
\]
on the subgroup scheme $A[r]$ of $r$-torsion points of $A$
with values in the $r$-th roots of unity.
If $\F$ contains $\zeta_r$,
we also have the non-trivial {\it Tate pairing}
\[
t_r: A[r](\F) \times A(\F)/rA(\F) \to \mathbf{F}^*/(\F^*)^r.
\]
The Weil and Tate pairings can be used to `embed' $r$-torsion subgroups of
$A(\F)$ into the multiplicative group $\F(\zeta_r)^*$, and thus the discrete
logarithm problem in $A(\F)[r]$ can be `reduced' to the same problem in
$\F(\zeta_r)^*$ \cite{menezes-okamoto-vanstone,frey-ruck}.
In pairing-based cryptographic protocols \cite{paterson},
one chooses the prime $r$ and the embedding
degree $k$ such that the discrete logarithm problems in $A(\F)[r]$ and
$\F(\zeta_r)^*$ are computationally infeasible, and
of roughly equal difficulty.
This means that $r$ is typically large, whereas $k$ is small.
Jacobians of curves meeting such requirements are often said to be
{\it pairing-friendly}.

If $\F$ has order $q$, the embedding degree
$k=[\F(\zeta_r):\F]$ is simply the multiplicative order of $q$ in $(\Z/r\Z)^*$.
As `most' elements in $(\Z/r\Z)^*$ have large order,
the embedding degree of $A$ with respect to a large prime divisor $r$
of $\#A(\F)$ will usually be of the same size as $r$, and $A$ will not
be pairing-friendly.
One is therefore led to the question of how to efficiently
construct $A$ and $\F$ such that $A(\F)$ has a (large) prime factor $r$
and the embedding degree of $A$ with respect to $r$ has a
prescribed (small) value $k$.
The current paper addresses this question on two levels:
the \emph{existence} and the actual \emph{construction} of $A$ and $\F$.

Section \ref{s:weilnumbers} focuses on the question whether,
for given $r$ and $k$,
there exist abelian varieties $A$ that are defined over a finite field $\F$,
have an $\F$-rational point of order $r$, and have embedding degree $k$
with respect to $r$.
We consider only abelian varieties $A$ that are \emph{simple},
that is, not isogenous (over $\F$) to a product of lower-dimensional varieties,
as we can always reduce to this case.
By Honda-Tate theory \cite{honda-tate}, isogeny classes of simple
abelian varieties $A$ over
the field $\F$ of $q$ elements are in one-to-one correspondence
with $\Gal(\Qbar/\Q)$-conjugacy classes of \emph{$q$-Weil numbers},
which are algebraic integers $\pi$ with the property that all
embeddings of $\pi$ into $\C$ have absolute value $\sqrt{q}$.
This correspondence is given by the map sending $A$
to its $q$-th power Frobenius endomorphism $\pi$ inside the number
field $\Q(\pi) \subset \End(A)\otimes \Q$.
The existence of abelian varieties with the properties we want is
thus tantamount to the existence of suitable Weil numbers.

Our main result, Algorithm \ref{alg:construct-pi},
constructs suitable $q$-Weil numbers $\pi$ in a given \emph{CM-field} $K$.
It exhibits $\pi$ as a \emph{type norm} of an element in a
\emph{reflex field} of $K$ satisfying certain congruences modulo $r$.
The abelian varieties $A$ in the isogeny classes over $\F$
that correspond to these Weil numbers have
an $\F$-rational point of order $r$ and embedding degree $k$
with respect to $r$.
Moreover, they are \emph{ordinary}, i.e.,
$\# A(\Fbar)[p]=p^g$, where $p$ is the characteristic of $\F$.
Theorem \ref{thm:ispoltime} shows that for fixed $K$,
the expected run time of our algorithm is heuristically
polynomial in $\log r$.

For an abelian variety of dimension $g$ over the field
$\F$ of $q$ elements, the group $A(\F)$ has roughly $q^g$ elements,
and one compares this size to $r$ by setting
\addtocounter{theorem}{1}
\begin{equation}\label{rho}
\rho=\frac{g\log q}{\log r}.
\end{equation}
In cryptographic terms, $\rho$ measures the ratio of a pairing-based system's
required bandwidth to its security level,
so small $\rho$-values are desirable.
{\it Supersingular} abelian varieties can achieve $\rho$-values close
to $1$, but their embedding degrees are limited to a few values that
are too small to be practical \cite{galbraith,rubin-silverberg}.
Theorem \ref{thm:heur} discusses
the distribution of the (larger) $\rho$-values we obtain.

In Section \ref{s:construct}, we address the issue of the actual construction
of abelian varieties corresponding to the Weil numbers found by
our algorithm.
This is accomplished via the construction in characteristic zero of the
abelian varieties having CM by the ring of integers $\O_K$ of $K$,
a hard problem that is far from being algorithmically
solved.
We discuss the elliptic case $g=1$,
for which reasonable
algorithms exist, and the case $g=2$, for which such algorithms
are still in their infancy.
For genus $g\ge3$, we restrict attention to a few families of curves
that we can handle at this point.
Our final Section \ref{s:examples} provides numerical examples.


\section{Weil numbers yielding prescribed embedding degrees}
\label{s:weilnumbers}

Let $\F$ be a field of $q$ elements,
$A$ a $g$-dimensional simple abelian variety over $\F$,
and $K=\Q(\pi) \subset \End(A)\otimes\Q$ the number
field generated by the Frobenius endomorphism $\pi$.
Then $\pi$ is a \emph{$q$-Weil number} in $K$:
an algebraic integer with the property that
all of its embeddings in $\Qbar$ have complex absolute value $\sqrt{q}$.

The $q$-Weil number $\pi$ determines the group
order of $A(\F)$: the $\F$-rational points of $A$ form the kernel of
the endomorphism $\pi-1$, and in the case where $K=\Q(\pi)$ is the full
endomorphism algebra $\End(A)\otimes\Q$ we have
$$
\#A(\F) = \N_{K/\Q}(\pi - 1).
$$
In the case $K=\End(A)\otimes\Q$ we will focus on,
$K$ is a \emph{CM-field} of degree $2g$ as in
\cite[Section 1]{honda-tate}, i.e.,
a totally complex quadratic extension of a totally
real subfield $K_0\subset K$.

\begin{proposition}
\label{p:exist-embed} Let $A$, $\F$ and $\pi$ be as above, and
assume $K=\Q(\pi)$ equals $\End_{\F}(A)\otimes \Q$.
Let $k$ be a positive integer, $\Phi_k$ the $k$-th cyclotomic polynomial,
and $r \nmid qk$ a prime number.
If we have
\begin{eqnarray*}
\N_{K/\Q}(\pi - 1) & \equiv & 0 \pmod{r}, \\
\Phi_k(\pi \pibar) & \equiv & 0 \pmod{r},
\end{eqnarray*}
then $A$ has embedding degree $k$ with respect to $r$.
\end{proposition}
\begin{proof}
The first condition tells us that $r$ divides $\#A(\F)$,
the second that the order of $\pi\pibar=q$ in $(\Z/r\Z)^*$,
which is the embedding degree of $A$ with respect to~$r$, equals $k$.
\qed
\end{proof}
By Honda-Tate theory \cite{honda-tate},
all $q$-Weil numbers arise as Frobenius elements of abelian varieties
over $\F$.
Thus, we can prove the \emph{existence}
of an abelian variety $A$ as in Proposition \ref{p:exist-embed}
by exhibiting a $q$-Weil number $\pi \in K$ as in that
proposition.
The following Lemma states what we need.
\begin{lemma}\label{l:honda-tate}
Let $\pi$ be a $q$-Weil number and $\F$ be the field of $q$ elements. 
Then there exists
a unique isogeny class of simple abelian varieties $A/\F$
with Frobenius $\pi$.
If $K=\Q(\pi)$ is totally imaginary of degree $2g$
and $q$ is prime, then such $A$ have dimension $g$, and $K$ is the full
endomorphism algebra $\End_{\F}(A)\otimes\Q$.
If furthermore $q$ is unramified in~$K$, then $A$ is ordinary.
\end{lemma}

\begin{proof}
The main theorem of \cite{honda-tate} yields
existence and uniqueness, and shows that
$E=\End_{\F}(A)\otimes \Q$ is a central simple algebra
over $K=\Q(\pi)$ satisfying
$$
2\cdot \mathrm{dim}(A)=[E:K]^{\frac{1}{2}}[K:\Q].
$$
For $K$ totally imaginary of degree $2g$ and $q$ prime,
Waterhouse \cite[Theorem 6.1]{waterhouse} shows that we have
$E=K$ and $\mathrm{dim}(A)=g$.
By \cite[Prop. 7.1]{waterhouse},
$A$ is ordinary if and only if $\pi + \pibar$ is prime to
$q=\pi\overline{\pi}$ in $\O_K$.
Thus if $A$ is not ordinary, the ideals
$(\pi)$ and $(\overline{\pi})$ have a common divisor $\p\subset \O_K$
with $\p^2 \mid q$, so $q$ ramifies in $K$.
\qed
\end{proof}

\begin{ex}\label{cycliccase}
Our general construction is motivated by the case where
$K$ is a Galois CM-field of degree $2g$, with cyclic Galois group
generated by~$\sigma$.
Here $\sigma^g$ is complex conjugation, so we
can construct an element $\pi\in\O_K$ satisfying
$\pi\sigma^g(\pi)=\pi\overline{\pi}\in\Z $
by choosing any $\xi\in \O_K$ and letting
$\pi=\prod_{i=1}^{g} \sigma^i(\xi)$.
For such $\pi$,
we have $\pi\overline{\pi}=N_{K/\Q }(\xi)\in\Z$.
If $N_{K/\Q }(\xi)$ is a prime $q$, then $\pi$ is a
$q$-Weil number in $K$.

Now we wish to impose the conditions of Proposition
\ref{p:exist-embed} on $\pi$. Let $r$ be a rational prime that
splits completely in $K$, and $\r$ a prime of
$\O_K$ over $r$.
For $i=1,\ldots,2g$, put $\r_i=\sigma^{-i}(\r )$;
then the factorization of $r$ in $\O_K$ is $r\O_K = \prod_{i=1}^{2g} \r_i$.
If $\alpha_i\in \F_r=\O_K/\r_i$ is the residue class of
$\xi$ modulo~$\r_i$, then $\sigma^i(\xi)$ modulo $\r$ is also $\alpha_i$,
so the residue class of $\pi$
modulo $\r$ is $\prod_{i=1}^{g}\alpha_i$.
Furthermore, the
residue class of $\pi\overline{\pi}$ modulo $\r$ is
$\prod_{i=1}^{2g}\alpha_i$.
If we choose $\xi$ to satisfy
\addtocounter{theorem}{1}
\begin{equation}\label{cond1}
\textstyle
\prod_{i=1}^g\alpha_i=1\in \F_r,
\end{equation}
we find $\pi\equiv 1 \pmod{\r}$ and thus $N_{K/\Q }(\pi-1)\equiv 0 \pmod{r}$.
By choosing $\xi$ such that in addition
\addtocounter{theorem}{1}
\begin{equation}\label{cond2}
\textstyle
\zeta=\prod_{i=1}^{2g}\alpha_i=\prod_{i=g+1}^{2g} \alpha_i
\end{equation}
is a primitive $k$-th root of unity in $\F_r^*$, we guarantee that
$\pi\overline{\pi}=q$ is a primitive $k$-th root of unity modulo
$r$. Thus we can try to find a Weil number as in Proposition
\ref{p:exist-embed} by picking residue classes $\alpha_i\in \F_r^*$
for $i=1,\ldots,2g$ meeting the two conditions above, computing some
`small' lift $\xi\in\O_K$ with $(\xi \bmod \r_i)=\alpha_i$, and
testing whether $\pi=\prod_{i=1}^{g} \sigma^i(\xi)$ has prime
norm. As numbers of moderate size have a high probability of being
prime by the prime number theorem, a small number of choices
$(\alpha_i)_i$ should suffice. There are $(r-1)^{2g-2}\varphi(k)$
possible choices for $(\alpha_i)_{i=1}^{2g}$,
where $\varphi$ is the Euler totient function,
so for $g>1$ and large $r$ we are very likely to succeed.
For $g=1$, there are only a few choices $(\alpha_1,\alpha_2)=(1,\zeta)$,
but one can try various lifts and thus recover what is known as the
Cocks-Pinch algorithm \cite[Theorem 4.1]{freeman-scott-teske} for
finding pairing-friendly elliptic curves.
\qed
\end{ex}

For arbitrary CM-fields $K$, the appropriate generalization of
the map
$$
\textstyle
\xi\mapsto \prod_{i=1}^{g}\sigma^i(\xi)
$$
in Example \ref{cycliccase}
is provided by the \emph{type norm}.
A \emph{CM-type} of a CM-field $K$ of degree $2g$
is a set $\Phi=\{\phi_1,\ldots,\phi_g\}$ of
embeddings of $K$ into its normal closure~$L$ such that
$\Phi\cup\overline\Phi=
\{\phi_1,\ldots,\phi_g,\overline{\phi_1},\ldots,\overline{\phi_g}\}$
is the complete set of embeddings of $K$ into~$L$.
The \emph{type norm} $N_{\Phi}: K\to L$ with respect to $\Phi$ is the map
$$
N_{\Phi}: x\longmapsto \textstyle{\prod_{i=1}^g\phi_i(x)},
$$
which clearly satisfies
\addtocounter{theorem}{1}
\begin{equation}\label{eq:typenorm}
N_{\Phi}(x)\overline{N_{\Phi}(x)}=N_{K/\Q }(x)\in\Q .
\end{equation}
If $K$ is not Galois, the type norm $N_\Phi$ does not map $K$ to
itself, but to its \emph{reflex field} $\widehat K$ with respect to
$\Phi$. To end up in $K$, we can however take the type norm with
respect to the \emph{reflex type} $\Psi$, which we will define now
(cf. \cite[Section 8]{shimura}).

Let $G$ be the Galois group of $L/\Q$, and $H$ the subgroup fixing $K$.
Then the $2g$ left cosets of $H$ in $G$ can be viewed as the embeddings
of $K$ in $L$, and this makes the CM-type $\Phi$ into a set of
$g$ left cosets of $H$ for which we have $G/H=\Phi\cup\overline\Phi$.
Let $S$ be the union of the left cosets in $\Phi$, and put
$\widehat S=\{\sigma^{-1}:\sigma\in S\}$.
Let $\widehat H = \{\gamma \in G : \gamma S = S\}$ be the stabilizer
of $S$ in $G$.
Then $\widehat H$ defines a subfield $\widehat K$ of $L$, and as we have
$\widehat H = \{\gamma \in G : \widehat S\gamma = \widehat S\}$
we can interpret $\widehat S$ as a
union of left cosets of $\widehat H$ inside $G$.
These cosets define a set of embeddings $\Psi$ of $\widehat K$ into~$L$.
We call $\widehat K$ the \emph{reflex field} of $(K,\Phi)$ and we
call $\Psi$ the {\it reflex type}.

\begin{lemma}
\label{l:reflex} The field $\widehat K$ is a CM-field. It is
generated over $\Q$ by the sums $\sum_{\phi\in \Phi}\phi(x)$ for
$x\in K$, and $\Psi$ is a CM-type of $\widehat K$. The type norm
$N_\Phi$ maps $K$ to~$\widehat K$.
\end{lemma}
\begin{proof}
The first two statements are proved in \cite[Chapter II, Proposition
28]{shimura} (though the definition of $\widehat H$ differs from
ours, because Shimura lets $G$ act from the right).
For the last statement, notice that for $\gamma\in \widehat H$, we
have $\gamma S=S$, so $\gamma\prod_{\phi\in
\Phi}\phi(x)=\prod_{\phi\in \Phi}\phi(x)$. \qed
\end{proof}

A CM-type $\Phi$ of $K$ is \emph{induced} from a CM-subfield
$K'\subset K$ if it is of the form
$\Phi=\{\phi : \phi|_{K'}\in \Phi'\}$ for some CM-type $\Phi'$ of $K'$.
In other words, $\Phi$ is induced from $K'$ if and only if $S$ as
above is a union of left cosets of $\mathrm{Gal}(L/K')$. We call
$\Phi$ {\it primitive} if it is not induced from a strict subfield of $K$;
primitive CM-types correspond to simple abelian varieties \cite{shimura}.
Notice that the reflex type $\Psi$ is primitive by
definition of $\widehat K$, and that $(K,\Phi)$ is induced from the
reflex of its reflex. In particular, if $\Phi$ is primitive, then
the reflex of its reflex is $(K,\Phi)$ itself.
For $K$ Galois and $\Phi$ primitive we have $\widehat K = K$, and the
reflex type of $\Phi$ is $\Psi = \{\phi^{-1} : \phi \in \Phi\}$.

For CM-fields $K$ of degree $2$ or $4$ with primitive CM-types,
the reflex field
$\widehat K$ has the same degree as $K$. This fails to be so for
$g\geq 3$.
\begin{lemma}
\label{l:reflexdeg}
If $K$ has degree $2g$, then the degree of $\widehat K$ divides $2^g g!$.
\end{lemma}
\begin{proof}
We have $K=K_0(\sqrt \eta)$, with $K_0$ totally real and $\eta\in K$
totally negative. The normal closure $L$ of $K$ is obtained by
adjoining to the normal closure $\widetilde {K_0}$ of $K_0$, which
has degree dividing $g!$, the square roots of the $g$ conjugates of
$\eta$. Thus $L$ is of degree dividing $2^g g!$, and $\widehat K$ is
a subfield of $L$.\qed
\end{proof}
For a `generic' CM field $K$ the degree of $L$ is exactly $2^gg!$,
and $\widehat{K}$ is a field of degree $2^g$ generated by
$\sum_\sigma \sqrt{\sigma(\eta)}$,
with $\sigma$ ranging over $\Gal(K_0/\Q)$.

{}From \eqref{eq:typenorm} and Lemma \ref{l:reflex}, we find that for
every $\xi\in \mathcal{O}_{\widehat K}$, the element
$\pi=N_{\Psi}(\xi)$ is an element of $\mathcal{O}_K$ that satisfies
$\pi\overline{\pi}\in\Z $. To make $\pi$ satisfy the conditions of
Proposition \ref{p:exist-embed}, we need to impose conditions modulo
$r$ on $\xi$ in~$\widehat K$. Suppose $r$ splits completely in $K$,
and therefore in its normal closure $L$ and in the reflex field
$\widehat K$ with respect to $\Phi$. Pick a prime $\R$ over $r$ in
$L$, and write $\r_\psi=\psi^{-1}(\R)\cap \O_{\widehat K}$ for
$\psi\in \Psi$.
Then the factorization of $r$ in $\O_{\widehat K}$ is
\addtocounter{theorem}{1}
\begin{equation}\label{rfac}
\textstyle
r\O_{\widehat{K}}=\prod_{\psi\in \Psi}
\r_\psi\overline{\r_\psi}.
\end{equation}

\begin{theorem}
\label{t:construct-pi}
Let $(K,\Phi)$ be a CM-type and $(\widehat K,\Psi)$ its reflex.
Let $r \equiv 1 \pmod{k}$ be a prime that splits completely in $K$,
and write its factorization in $\O_{\widehat K}$ as in \eqref{rfac}.
Given $\xi\in\mathcal{O}_{\widehat K}$,
write $(\xi\bmod \r_\psi)=\alpha_\psi\in \F_r$ and
$(\xi\bmod \overline{\r_\psi})=\beta_\psi\in \F_r$ for $\psi\in\Psi$.
If we have
\addtocounter{theorem}{1}
\begin{equation}
\label{eq:conditions}
\textstyle
\prod_{\psi\in\Psi} \alpha_\psi=1
\qquad\mathrm{and}\qquad
\prod_{\psi\in\Psi} \beta_\psi=\zeta
\end{equation}
for some primitive $k$-th root of unity $\zeta\in \F_r^*$,
then $\pi=N_{\Psi}(\xi)\in\O_K$ satisfies $\pi\overline{\pi}\in\Z $ and
\begin{eqnarray*}
\N_{K/\Q}(\pi - 1) & \equiv & 0 \pmod{r}, \\
\Phi_k(\pi \pibar) & \equiv & 0 \pmod{r}.
\end{eqnarray*}
\end{theorem}

\begin{proof}
This is a straightforward generalization of the argument in
Example \ref{cycliccase}.
The conditions \eqref{eq:conditions} generalize \eqref{cond1}
and \eqref{cond2}, and imply in the present context
that $\pi-1\in\O_K$ and $\Phi_k(\pi\pibar)\in\Z$ are in the prime
$\R\subset\O_L$ over $r$ that underlies the factorization \eqref{rfac}.
\qed
\end{proof}
If the element $\pi$ in Theorem \ref{t:construct-pi} generates $K$
and $\N_{K/\Q}(\pi)$ is a prime $q$ that is unramified in $K$,
then by Lemma \ref{l:honda-tate} $\pi$ is a
$q$-Weil number corresponding to
an ordinary abelian variety $A$ over $\F=\F_q$ with endomorphism algebra
$K$ and Frobenius element $\pi$.
By Proposition \ref{p:exist-embed}, $A$ has embedding degree $k$
with respect to $r$.
This leads to the following algorithm.
\begin{algorithm}
\label{alg:construct-pi}
\hfil\break\indent
Input: a CM-field $K$ of degree $2g \ge 4$,
a primitive CM-type $\Phi$ of $K$,
a positive integer $k$,
and a prime $r \equiv 1 \pmod{k}$ that splits completely in $K$.

Output: a prime $q$ and a $q$-Weil number $\pi\in K$
corresponding to an ordinary, simple abelian variety $A/\F_q$ with embedding
degree $k$ with respect to $r$.

\begin{enumerate}
\item Compute a Galois closure $L$ of $K$ and the reflex
$(\widehat K,\Psi)$ of $(K,\Phi)$.
Set $\widehat g \leftarrow \frac{1}{2} \deg \widehat K$
and write $\Psi=\{\psi_1, \psi_2,\ldots, \psi_\ghat\}$.
\item \label{st:factor}
Fix a prime $\mathfrak{R}\mid r$ of $\O_L$, and compute
the factorization of $r$ in $\O_{\widehat K}$
as in \eqref{rfac}.
\item
Compute a primitive $k$-th root of unity $\zeta\in \F_r^*$.
\item \label{st:random}
Choose random
$\alpha_1,\ldots,\alpha_{\widehat{g}-1},\beta_1,\ldots,\beta_{\widehat{g}-1}
\in \F_r^*$.
\item
Set $\alpha_{\widehat{g}} \leftarrow \prod_{i=1}^{\widehat{g}-1}
\alpha_i^{-1} \in \F_r^*$ and $\beta_{\widehat{g}} \leftarrow \zeta
\prod_{i=1}^{\widehat{g}-1} \beta_i^{-1} \in \F_r^*$.
\item \label{st:crt}
Compute $\xi \in \O_{\widehat K}$ such that
$(\xi\mod \r_{\psi_i})=\alpha_i$ and
$(\xi\mod\overline{\r_{\psi_i}})=\beta_i$ for $i=1,2,\ldots , \widehat g$.
\item \label{st:q}
Set $q \leftarrow \N_{\widehat K/\Q}(\xi)$.
If $q$ is not prime, go to Step \eqref{st:random}.
\item \label{st:pi}
Set $\pi \leftarrow N_{\Psi}(\xi)$.
If $q$ is not unramified in $K$, or $\pi$ does not generate $K$,
go to Step \eqref{st:random}.
\item \label{st:output}
Return $q$ and $\pi$.
\end{enumerate}
\end{algorithm}

\begin{rem}
We require $g\ge 2$ in Algorithm \ref{alg:construct-pi}, as the case
$g=1$ is already covered by Example \ref{cycliccase}, and requires
a slight adaptation.

The condition that $r$ be prime is for simplicity of presentation
only;
the algorithm easily extends to square-free values of $r$ that are
given as products of splitting primes.
Such $r$ are required, for example, by the cryptosystem of
\cite{boneh-goh-nissim}.
\end{rem}

\section{Performance of the algorithm}

\begin{theorem}
\label{thm:ispoltime}
If the field $K$ is fixed, then
the heuristic expected run time of Algorithm \ref{alg:construct-pi}
is polynomial
in~$\log r$.
\end{theorem}
\begin{proof}
The algorithm consists of a precomputation for the
field $K$ in Steps (1)--(3), followed by a loop in Steps
\eqref{st:random}--\eqref{st:q} that is performed until an element
$\xi$ is found that has prime norm $N_{\widehat K/\Q}(\xi)=q$,
and we also find in Step \eqref{st:pi} that
$q$ is unramified in $K$ and the type norm $\pi=N_\Psi(\xi)$ generates $K$.

The primality condition in Step \eqref{st:q} is the `true'
condition that becomes harder to achieve with increasing $r$,
whereas the conditions in Step \eqref{st:pi}, which are
necessary to guarantee correctness of the output, are so
extremely likely to be fulfilled
(especially in cryptographic applications where $K$ is small and $r$ is large)
that they will hardly ever fail in practice
and only influence the run time by a constant factor.

As $\xi$ is computed in Step \eqref{st:crt} as the lift to $\O_{\widehat K}$
of an element $\overline{\xi}\in\O_{\widehat K}/r\O_{\widehat K}\cong
(\F_r)^{2\widehat g}$,
its norm can be bounded by a constant multiple of $r^{2\widehat{g}}$.
Heuristically,  $q=\N_{\widehat{K}/\Q}(\xi)$ behaves as a random number, so
by the prime number theorem it will be prime with probability
at least $(2\widehat{g}\log r)^{-1}$,
and we expect that we need to repeat the loop in Steps
\eqref{st:random}--\eqref{st:q} about $2\widehat{g}\log r$ times
before finding $\xi$ of prime norm $q$.
As each of the steps is polynomial in $\log r$, so is the expected run time
up to Step \eqref{st:q}, and we are done if we show that the conditions
in Step \eqref{st:pi} are met with some positive probability if
$K$ is fixed and $r$ is sufficiently large.

For $q$ being unramified in $K$, one simply notes that
only finitely many primes ramify in the field $K$ (which is fixed)
and that $q$ tends to infinity with $r$,
since $r$ divides $\N_{K/\Q}(\pi -1) \le (\sqrt{q}+1)^{2g}$.

Finally, we show that $\pi$ generates $K$ with probability
tending to $1$ as $r$ tends to infinity.
Suppose that for every vector $v\in\{0,1\}^{\ghat}$ that is not all 0 or 1,
we have \addtocounter{theorem}{1}
\begin{equation}\label{e:pigeneratesk}
\textstyle \prod_{i=1}^{\ghat}(\alpha_i/\beta_i)^{v_i}\not=1.
\end{equation}
This set of $2^{\ghat}-2$ (dependent) conditions on the
$2\ghat-2$ independent random variables $\alpha_i, \beta_i$ for
$1 \le i<\ghat$
is satisfied with probability at least $1-(2^{\ghat}-2)/(r-1)$.
For any automorphism $\phi$ of $L$, the set
$\phi\circ \Psi$ is
a CM-type of $\widehat{K}$ and there is a $v\in\{0,1\}^{\ghat}$ such
that $v_i = 0$ if  $\phi\circ\Psi$ contains $\psi_i$ and $v_i=1$
otherwise.
Then $\alpha_i$ is $(\psi_i(\xi)\bmod\R)$, while $\beta_i$ is
$(\overline{\psi_i(\xi)}\bmod \R)$, so $(\pi/\phi(\pi)\bmod \R)$ is
$\prod_{i=1}^{\ghat}(\alpha_i/\beta_i)^{v_i}$.  By
\eqref{e:pigeneratesk}, if this expression is $1$
then $v=0$ or $v=1$, so $\phi\circ\Psi=\Psi$ or
$\overline{\phi}\circ\Psi=\Psi$, which by definition of
the reflex is equivalent to $\phi$ or $\overline{\phi}$
being trivial on $K$, i.e., to $\phi$ being trivial on
the maximal real subfield $K_0$. Thus if
\eqref{e:pigeneratesk} holds, then $\phi(\pi) = \pi$ implies that
$\phi$ is trivial on $K_0$, hence $K_0\subset \Q(\pi)$.
Since $\pi\in K$ is not real
(otherwise, $q=\pi^2$ ramifies in $K$),
this implies that $K = \Q(\pi)$.\qed
\end{proof}

In order to maximize the likelihood of finding prime norms, one
should minimize the norm of the lift $\xi$ computed in
the Chinese Remainder Step \eqref{st:crt}.
This involves minimizing a norm function of degree $2\widehat{g}$
in $2\widehat{g}$ integral variables, which is already
infeasible for $\widehat{g}=2$.

In practice, for given $r$, one lifts a standard basis of
$\O_{\widehat K}/r\O_{\widehat K}\cong (\F_r)^{2\widehat g}$
to $\O_{\widehat K}$. Multiplying those
lifts by integer representatives for the elements $\alpha_i$ and
$\beta_i$ of $\F_r$,
one quickly obtains lifts $\xi$.
We also choose, independently of $r$, a $\Z$-basis of
$\O_{\widehat{K}}$ consisting
of elements that are `small' with respect to all absolute values of
$\widehat{K}$.
We translate $\xi$ by multiples of $r$ to lie in $rF$, where $F$ is the
fundamental parallelotope in $\widehat K \otimes \RR$ consisting of
those elements that have coordinates in $(-\frac{1}{2}, \frac{1}{2}]$
with respect to our chosen basis.

If we denote the maximum on $F\cap \widehat K$ of all complex
absolute values of $\widehat K$ by $M_{\widehat{K}}$, we have
$
q=N_{\widehat{K}/\Q}(\xi)\leq (rM_{\widehat{K}})^{2\widehat{g}}.
$
For the $\rho$-value \eqref{rho}
we find
\addtocounter{theorem}{1}
\begin{equation}\label{eqn:upperboundrho}
\rho\leq 2g\widehat{g}(1+\log{M_{\widehat{K}}}/\log{r}),
\end{equation}
which is approximately $2g\ghat$ if $r$ gets large with respect to
$M_{\widehat{K}}$.
We would like $\rho$ to be small, but this is not what one obtains
by lifting random admissible choices of $\overline{\xi}$.
\begin{theorem}\label{thm:heur}
If the field $K$ is fixed and $r$ is large, we expect that (1) the output $q$
of Algorithm \ref{alg:construct-pi} yields $\rho\approx 2g\ghat$,
and (2) an optimal choice of $\xi\in\O_{\widehat{K}}$
satisfying the conditions of Theorem \ref{t:construct-pi}
yields $\rho\approx 2g$.
\end{theorem}
\begin{open}
Find an efficient algorithm to compute an element $\xi\in\O_{\widehat{K}}$
satisfying the conditions of Theorem \ref{t:construct-pi} for which
$\rho\approx 2g$.
\end{open}
We will prove Theorem \ref{thm:heur} via a series of lemmas.
Let $H_{r,k}$ be the subset of the parallelotope
$rF\subset \widehat{K}\otimes\RR$ consisting of
those $\xi\in rF\cap \O_{\widehat{K}}$ that satisfy the two
congruence conditions \eqref{eq:conditions}
for a given embedding degree $k$.
Heuristically, we will treat the elements of $H_{r,k}$ as random elements of
$rF$ with respect to the distributions of
complex absolute values and norm functions.
We will also use the fact that, as $\widehat K$ is totally complex
of degree $2\widehat g$, the
$\RR$-algebra $\widehat{K}\otimes\RR$ is naturally isomorphic to
$\C^{\widehat{g}}$.  We assume throughout that $g \ge 2$.

\begin{lemma}
\label{l:lower-max}
   Fix the field $K$. Under our heuristic assumption,
there exists a constant $c_1>0$ such that for all
$\varepsilon>0$, the probability
that a random $\xi\in H_{r,k}$ satisfies $q<r^{2(\ghat-\varepsilon)}$
is less than $c_1 r^{-\varepsilon}$.
\end{lemma}
\begin{proof}
   The probability that a random $\xi$ lies in the set
   $V=\{z\in\C^\ghat: \prod \abs{z_i}^2\leq r^{2(\ghat -\varepsilon)}\}\cap rF$
   is the quotient of the volume of $V$ by the volume
   $2^{-\ghat}\sqrt{\abs{\Delta_{\widehat{K}}}}r^{2\ghat}$
   of $rF$, where $\Delta_{\widehat K}$ is the discriminant of $\widehat K$.
   Now $V$ is contained inside
   $W=\{z\in \C^\ghat: \prod \abs{z_i}^2\leq r^{2(\ghat -\varepsilon)},
     \abs{z_i}\leq rM_{\widehat{K}}\}$,
   which has volume
   $$(2\pi)^\ghat \mathop{\int}_{\hbox to0pt{\hss $\scriptstyle\substack{
x\in [0,rM_{\widehat{K}}]^\ghat \\  \prod \abs{x_i}^2\leq r^{2(\ghat -\varepsilon)}
     }$}} \prod|x_i| dx
   \ \   <\ \ (2\pi)^\ghat
   \mathop{\int}_{\hbox to0pt{\hss $\scriptstyle x\in
[0,rM_{\widehat{K}}]^\ghat$}}
  r^{\ghat-\varepsilon}dx
\ \ =
\ \  (2\pi M_{\widehat{K}})^\ghat r^{2\ghat-\varepsilon},$$
   so a random $\xi$ lies in $V$ with probability less than
   $(4\pi M_{\widehat{K}})^\ghat\abs{\Delta_{\widehat{K}}}^{-1/2}
r^{-\varepsilon}$.\qed
\end{proof}

\begin{lemma}\label{lem:aux}
There exists a number $Q_{\widehat{K}}$, depending only on $\widehat{K}$,
such that for any positive real number $X<rQ_{\widehat{K}}$,
the expected number of
$\xi\in H_{r,k}$ with all absolute values below $X$ is
$$\frac{\varphi(k)(2\pi)^{\widehat{g}}}
{\abs{\Delta_{\widehat K}}}\frac{X^{2\ghat}}{r^{2}}.$$
\end{lemma}
\begin{proof}
Let $Q_{\widehat{K}}>0$ be a lower bound on
$\widehat{K} \setminus F$ for the maximum of all complex absolute
values, so the box $V_X\subset\widehat{K}\otimes\RR$ consisting of
those elements that have all absolute values below $X$ lies
completely inside $(X/Q_{\widehat{K}})F \subset rF$.
  The volume of
$V_X$ in $\widehat{K}\otimes\RR$ is $(\pi X^2)^{\widehat{g}}$, while
$rF$ has volume $2^{-\ghat}\sqrt{\abs{\Delta_{\widehat
K}}}r^{2\ghat}$.
The expected number of $\xi\in H_{r,k}$ satisfying $\abs{\xi}<X$
for all absolute values is
$\#H_{r,k}=r^{2\widehat{g}-2}\varphi(k)$ times the quotient of these
volumes.
\qed\end{proof}

\begin{lemma}
\label{l:upper-min}
Fix the field $K$. Under our heuristic assumption,
there exists a constant $c_2$ such that for all positive
$\varepsilon<2\ghat-2$,
if $r$ is sufficiently large, then we expect the number of
$\xi\in H_{r,k}$ satisfying $N_{\widehat{K}/\Q}(\xi)<r^{2+\varepsilon}$
to be at least $c_2r^\varepsilon$.
\end{lemma}
\begin{proof}
Any $\xi$ as in Lemma \ref{lem:aux} satisfies
$N_{\widehat{K}/\Q}(\xi)<X^{2\ghat}$, so we apply the lemma
to $X=r^{(1/\ghat+\varepsilon/2\ghat)}$, which is less than $rQ_{\widehat{K}}$
for large enough $r$ and $\epsilon<2\ghat-2$. \qed
\end{proof}

\begin{lemma}
\label{l:lower-min}
Fix the field $K$. Under our heuristic assumption,
for all $\varepsilon>0$,
if $r$ is large enough, we expect there to be no
$\xi\in H_{r,k}$ satisfying $N_{\widehat{K}/\Q}(\xi)<r^{2-\varepsilon}$.
\end{lemma}
\begin{proof}
Let $\widehat{\mathcal{O}}$ be the ring of integers of the maximal
real subfield
of $\widehat{K}$. Let $U$ be the subgroup of norm one elements of
$\widehat{\mathcal{O}}^*$.
We embed $U$ into
$\RR^{\ghat}$ by mapping $u\in U$ to the vector $l(u)$ of logarithms
of absolute
values of $u$. The image is a complete lattice in the
$(\ghat-1)$-dimensional space of vectors with coordinate sum $0$.
Fix a fundamental parallelotope $F'$ for this lattice.
Let $\xi_0$ be the element of $H_{r,k}$ of smallest norm.
Since the conditions \eqref{eq:conditions},
as well as the norm of $\xi_0$, are invariant
under multiplication by elements of $U$, we may assume
without loss of generality that $l(\xi_0)$ is
inside $F'+\C(1,\ldots,1)$.
Then every difference of two entries of $l(\xi_0)$ is bounded,
and hence every quotient of absolute values of $\xi_0$ is bounded from
below by a positive constant $c_3$ depending only on $K$. In
particular, if $m$ is the maximum of all absolute values of
$\xi_0$, then $N_{\widehat{K}/\Q}(\xi)>(c_3m)^{2\ghat}$.  Now suppose
$\xi_0$ has norm below $r^{2-\varepsilon}$.  Then all absolute
values of $\xi_0$ are below $X=r^{(1/\ghat-\varepsilon/2\ghat)}/c_3$,
and $X < rQ_{\widehat{K}}$ for $r$ sufficiently large.
Now Lemma \ref{lem:aux} implies that
the expected number of $\xi \in H_{r,k}$ with all absolute values
below $X$ is a constant times $r^{-\varepsilon}$,
so for any sufficiently large $r$ we expect there to be no such
$\xi$, a contradiction.
\qed
\end{proof}

\begin{proof}[of Theorem {\ref{thm:heur}}]
The upper bound $\rho \apprle 2g\ghat$ follows from \eqref{eqn:upperboundrho}.
Lemma \ref{l:lower-max} shows that for any $\varepsilon>0$, the
probability that $\rho$ is smaller than
$2g\ghat-\varepsilon$ tends to zero as $r$ tends to infinity,
thus proving the lower bound $\rho \apprge 2g\ghat$.
Lemma~\ref{l:upper-min} shows that for any $\varepsilon>0$, if $r$ is
sufficiently
large then we expect there to exist a $\xi$ with $\rho$-value at most
$2g+\varepsilon$,
thus proving the bound $\rho \apprle 2g$.
Lemma~\ref{l:lower-min} shows that
we expect $\rho>2g-\varepsilon$ for the optimal $\xi$,
which proves the bound $\rho\apprge 2g$.\qed
\end{proof}
For very small values of $r$ we are able to do a brute-force
search for the smallest $q$ by testing all possible values of $\alpha_1,\ldots,
\alpha_{\widehat{g}-1},\beta_1,\ldots,\beta_{\widehat{g}-1}$ in Step
\ref{st:random} of Algorithm
\ref{alg:construct-pi}.  We performed two such searches, one in dimension
2 and one in dimension 3.
The experimental results support our heuristic evidence that
$\rho\approx 2g$ is
possible
with a smart choice in the algorithm, and that $\rho\approx 2g\widehat{g}$ is
achieved with a randomized
algorithm.

\begin{ex}
Take $K = \Q(\zeta_5)$, and let $\Phi=\{\phi_1, \phi_2\}$ be the CM-type
of $K$ defined by $\phi_n (\zeta_5) = e^{2\pi i n/5}$.
We ran Algorithm \ref{alg:construct-pi} with $r = 1021$
and $k = 2$, and tested all possible values of $\alpha_1,\beta_1$.
The total number of primes $q$ found was $125578$, and the
corresponding $\rho$-values were distributed as follows:
\[
\includegraphics[width=5cm]{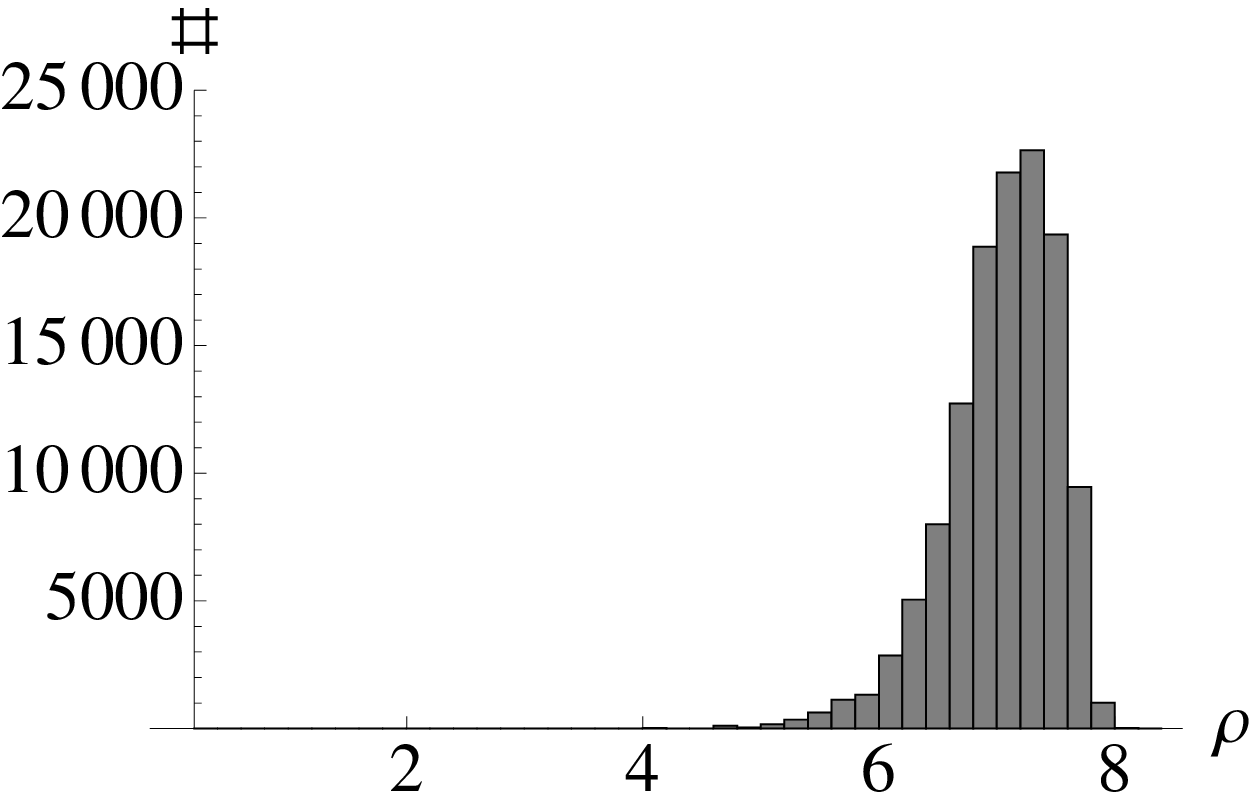}\quad\includegraphics[width=4.5cm]{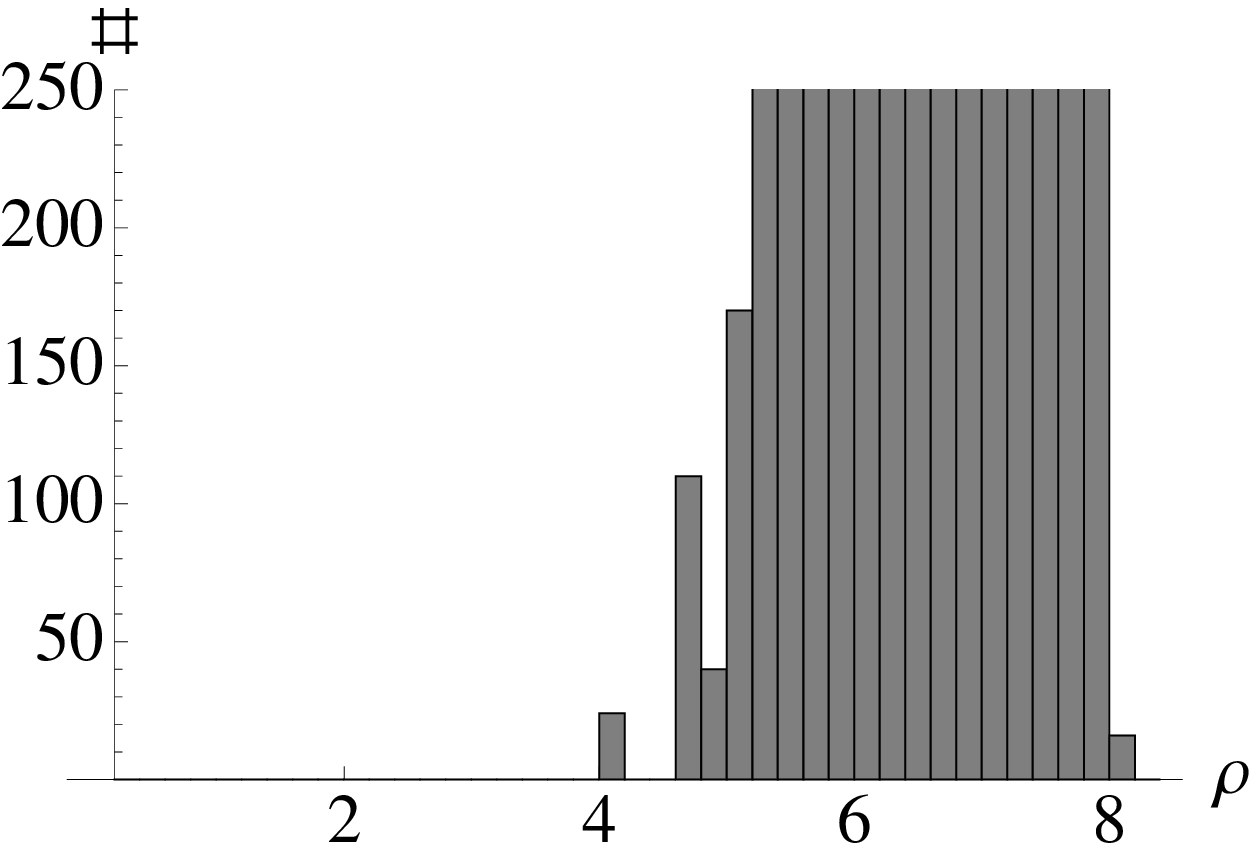}
\]
The smallest $q$ found was $2023621$, giving a $\rho$-value of $4.19$.
The curve over $\F=F_q$ for which the Jacobian has this $\rho$-value is
$y^2 = x^5 + 18$, and the number of points on its Jacobian is $4092747290896$.
\end{ex}

\begin{ex}
Take $K = \Q(\zeta_7)$, and let $\Phi=\{\phi_1, \phi_2, \phi_3\}$
be the CM-type of $K$ defined by $\phi_i (\zeta_7) = e^{2\pi i/7}$.
We ran Algorithm \ref{alg:construct-pi} with $r = 29$
and $k = 4$, and tested all possible values of
$\alpha_1,\alpha_2,\beta_1,\beta_2$.
The total number of primes $q$ found was $162643$, and the
corresponding $\rho$-values were distributed as follows:
\[
\includegraphics[width=5cm]{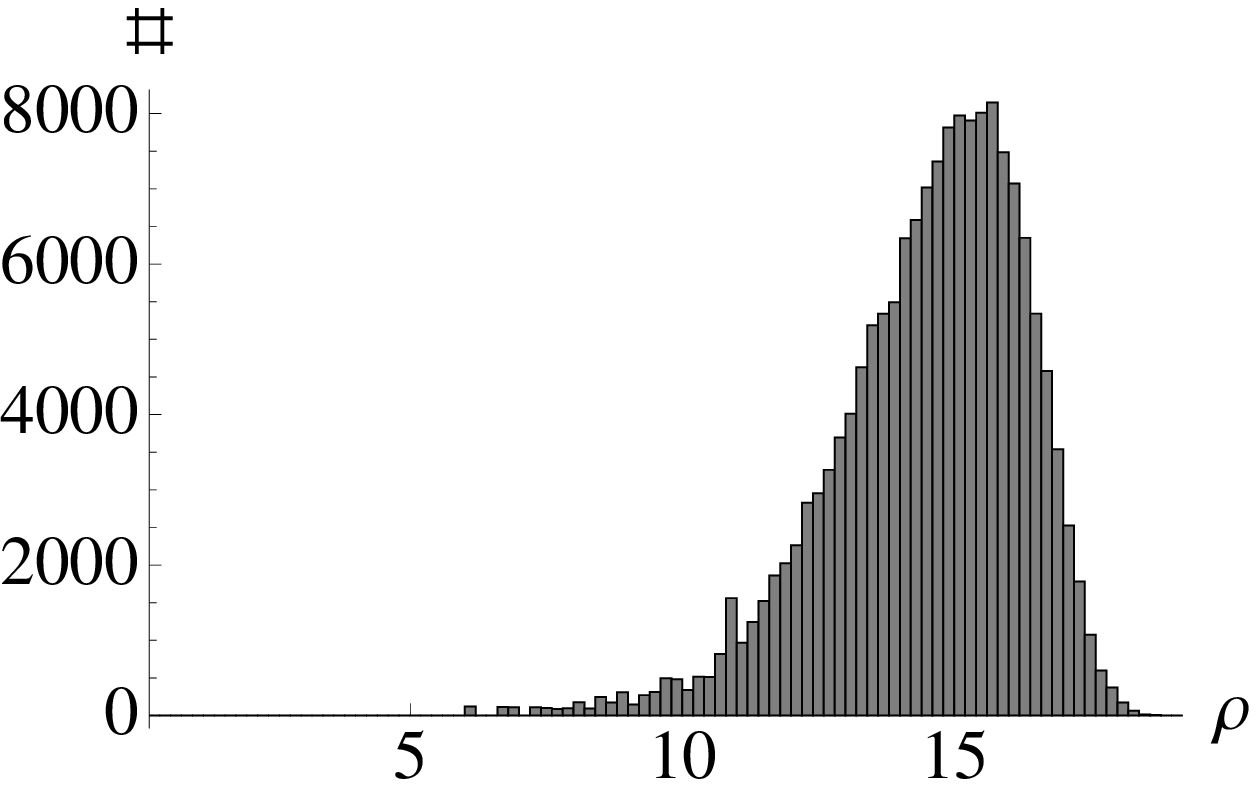}\quad\includegraphics[width=4.5cm]{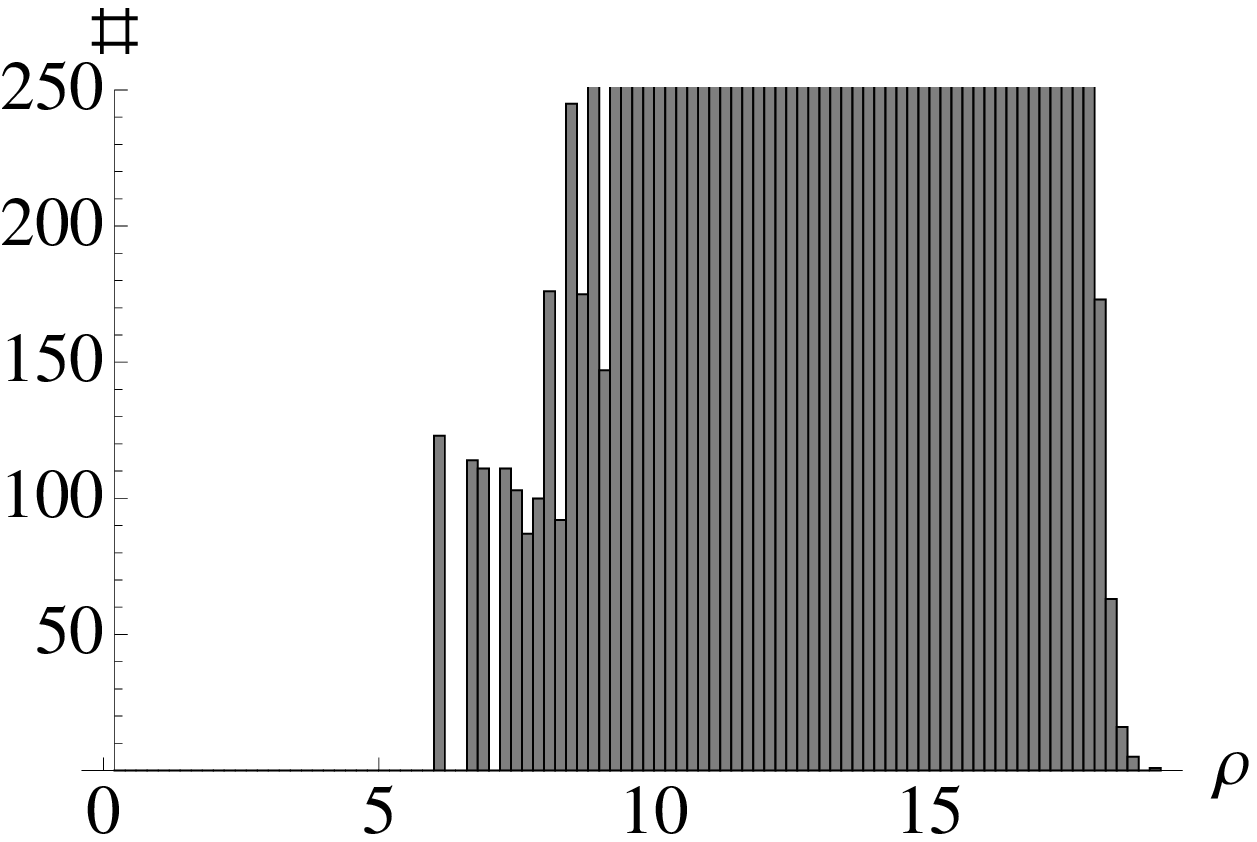}
\]
The smallest $q$ found was $911$, giving a $\rho$-value of $6.07$.
The curve over $\F=F_q$ for which the Jacobian has this $\rho$-value is
$y^2 = x^7 + 34$, and the number of points on its Jacobian is $778417333$.
\end{ex}

\begin{ex}
Take $K = \Q(\zeta_5)$, and let $\Phi=\{\phi_1, \phi_2\}$ be the
CM-type of $K$ defined by $\phi_i (\zeta_5) = e^{2\pi i/5}$. We ran
Algorithm \ref{alg:construct-pi} with $r = 2^{160}+685$ and $k =
10$, and tested $2^{20}$ random values of $\alpha_1,\beta_1$. The
total number of primes $q$ found was $7108$.
Of these primes, 6509 (91.6\%) produced $\rho$-values between 7.9 and
8.0, while 592 (8.3\%) had $\rho$-values between 7.8 and 7.9.
The smallest $q$ found had $623$ binary digits, giving a
$\rho$-value of $7.78$.
\end{ex}


\section{Constructing abelian varieties with given Weil numbers}
\label{s:construct}

Our Algorithm \ref{alg:construct-pi} yields $q$-Weil numbers
$\pi\in K$
that correspond, in the sense of
Honda and Tate \cite{honda-tate}, to isogeny classes of ordinary,
simple abelian varieties over prime fields
that have a point of order $r$ and embedding degree $k$
with respect to $r$.
It does not give a method to explicitly construct an
abelian variety $A$ with Frobenius $\pi\in K$.
In this section we focus on the problem of explicitly
constructing such varieties using complex multiplication techniques.

The key point of the complex multiplication construction is the fact that every
ordinary, simple abelian variety over $\F=\F_q$ with Frobenius $\pi\in K$
arises as the reduction at a prime over $q$ of some abelian variety $A_0$
in characteristic zero that has CM by the ring of integers of $K$.
Thus if we have fixed our $K$ as in Algorithm \ref{alg:construct-pi},
we can solve the construction problem for all ordinary
Weil numbers coming out
of the algorithm by compiling the finite list of
$\overline{\Q}$-isogeny classes of
abelian varieties in characteristic zero having CM by $\O_K$.
There will be one $\overline{\Q}$-isogeny class for
each equivalence class of primitive CM-types
of $K$, where $\Phi$ and $\Phi'$ are said to be equivalent
if we have $\Phi=\Phi'\circ \sigma$ for an automorphism $\sigma$
of $K$.
As we can choose our favorite field $K$ of degree $2g$ to
produce abelian varieties of dimension $g$, we can pick fields $K$
for which such lists already occur in the literature.

{}From representatives of our list of isogeny classes of abelian
varieties in characteristic zero having CM by $\O_K$, we obtain a
list $\mathcal{A}$ of abelian varieties over $\F$ with CM by
$\O_K$ by reducing at some fixed prime $\q$ over $q$. Changing the
choice of the prime $\q$ amounts to taking the reduction at $\q$ of
a conjugate abelian variety, which also has CM by $\O_K$ and hence is
$\Fbar$-isogenous to one already in the list.

For every abelian variety $A\in \mathcal{A}$,
we compute the set of its twists,
i.e., all the varieties up to $\F$-isomorphism that
become isomorphic to $A$ over $\Fbar$.
There is at least one twist $B$ of an element $A\in\mathcal{A}$
satisfying $\#B(\F) = \N_{K/\Q}(\pi-1)$, and this $B$ has a point of
order $r$ and the desired embedding degree.

Note that while efficient point-counting algorithms do not exist for varieties
of dimension $g > 1$, we can determine probabilistically whether an abelian
variety has a given order by choosing a random point, multiplying by the
expected order, and seeing if the result is the identity.

The complexity of the construction problem rapidly increases
with the genus $g=[K:\Q]/2$, and it is fair to say that we only
have satisfactory general methods at our disposal in very small genus.

In genus one, we are dealing with elliptic curves.
The $j$-invariants of elliptic curves over $\C$ with CM by $\O_K$
are the roots of the {\it Hilbert class polynomial} of~$K$,
which lies in $\Z[X]$.
The degree of this polynomial is the class number $h_K$ of~$K$, and
it can be computed in time $\widetilde O(\abs{\Delta_K})$.

For genus 2, we have to construct abelian surfaces.
Any principally polarized abelian surface is the Jacobian of a genus 2 curve,
and all genus 2 curves are hyperelliptic.  There is a theory of class
polynomials analogous to that for elliptic curves, as well as several
algorithms
to compute these polynomials, which lie in $\Q[X]$.
The genus 2 algorithms are not as well-developed as those for elliptic curves;
at present they can handle only very small quartic CM-fields, and there
exists no rigorous run time estimate.
{}From the roots in $\F$ of these polynomials, we can compute
the genus 2 curves using Mestre's algorithm.

Any three-dimensional principally polarized abelian
variety is isogenous to the Jacobian of a genus 3 curve.
There are two known families of
genus 3 curves over $\C$ whose Jacobians have CM by an order of dimension $6$.
The first family, due to Weng \cite{weng-g3}, gives hyperelliptic curves whose
Jacobians have CM by a degree-6 field containing $\Q(i)$.  The second family,
due to Koike and Weng \cite{koike-weng}, gives Picard curves (curves
of the form
$y^3 = f(x)$ with $\deg f = 4$) whose Jacobians have CM by a degree-6 field
containing $\Q(\zeta_3)$.

Explicit CM-theory is mostly undeveloped for
dimension $\ge 3$. Moreover, most principally polarized abelian varieties of
dimension $\ge 4$ are not Jacobians, as
the moduli space of Jacobians has dimension $3g -3$,
while the moduli space of abelian varieties has dimension $g(g+1)/2$.
For implementation purposes we prefer Jacobians or even hyperelliptic
Jacobians,
as these are the only abelian varieties for which group operations can be
computed efficiently.

In cases where we cannot compute every abelian
variety in characteristic zero with CM
by $\O_K$, we use a single such variety $A$
and run Algorithm \ref{alg:construct-pi} for each different CM-type
of $K$ until it yields a prime $q$ for which the reduction
of $A$ mod $q$ is in the correct isogeny class.
An example for $K=\Q(\zeta_{2p})$ with $p$ prime
is given by the Jacobian of $y^2 = x^p+a$, which has dimension $g = (p-1)/2$.


\section{Numerical examples}
\label{s:examples}

We implemented Algorithm \ref{alg:construct-pi} in MAGMA and used it
to compute examples of hyperelliptic curves of genus 2 and 3 over
fields of cryptographic size for which the Jacobians are
pairing-friendly. The subgroup size $r$ is chosen so that the
discrete logarithm problem in $A[r]$ is expected to take roughly
$2^{80}$ steps.  The embedding degree $k$ is chosen so that $r^{k/g}
\approx 1024$; this would be the ideal embedding degree for the
80-bit security level if we could construct varieties over $\F = \F_q$ with
$\#A(\F) \approx r$.  Space constraints prevent us from giving the
group orders for each Jacobian, but we note that a set of all
possible $q$-Weil numbers in $K$, and hence all possible group
orders, can be computed from the factorization of $q$ in $K$.

\begin{ex}
Let $\eta = \sqrt{-2+\sqrt{2}}$ and let $K$ be the degree-4 Galois CM field $\Q
(\eta)$.  Let $\Phi=\{\phi_1, \phi_2\}$ be the CM type of $K$ such
that $\Img(\phi_i(\eta)) > 0$.
We ran Algorithm \ref{alg:construct-pi} with CM type $(K,\Phi)$, $r = 2^{160} -
1679$, and $k = 13$.  The algorithm output the following field size:
{\scriptsize
\begin{eqnarray*}
q & = &
31346057808293157913762344531005275715544680219641338497449500238872300350617165
\setminus\\&&
40892530853973205578151445285706963588204818794198739264123849002104890399459807
\setminus\\&&
463132732477154651517666755702167 \quad \mbox{(640 bits)}
\end{eqnarray*}}
\noindent
There is a single $\Fbar_q$-isomorphism class of curves over $\F_q$ whose
Jacobians have CM by $\O_K$ and it has been computed in \cite{vanwamelen};
the desired twist turns out to be
$C: y^2 = -x^5 + 3x^4 + 2x^3 - 6x^2 - 3x + 1$.
The $\rho$-value of $\Jac(C)$ is $7.99$.
\end{ex}

\begin{ex}
Let $\eta = \sqrt{-30+2\sqrt{5}}$ and let $K$ be the degree-4 non-Galois CM
field $\Q(\eta)$.  The reflex field $\widehat K$ is $\Q(\omega)$
where $\omega = \sqrt
{-15+2\sqrt{55}}$.  Let $\Psi$ be the CM type of $K$ such that $\Img(\phi_i
(\eta)) > 0$.  We ran Algorithm \ref{alg:construct-pi} with the CM type $(K,
\Phi)$, subgroup size $r = 2^{160} - 1445$, and embedding degree $k = 13$.  The
algorithm output the following field size:
{\scriptsize
\begin{eqnarray*}
q & = &
11091654887169512971365407040293599579976378158973405181635081379157078302130927
\setminus\\&&
51652003623786192531077127388944453303584091334492452752693094089192986541533819
\setminus\\&&
35518866167783400231181308345981461 \quad \mbox{(645 bits)}
\end{eqnarray*}}
\noindent
The class polynomials for $K$ can be found in the preprint version of
\cite{weng-g2}.  We used the roots of the class polynomials mod $q$ to
construct curves over $\F_q$ with CM by $\O_K$.  As $K$ is non-Galois
with class number $4$, there are 8 isomorphism classes
of curves in 2 isogeny classes.  We found a curve $C$ in the correct isogeny
class with equation $y^2 = x^5 + a_3 x^3 + a_2 x^2 + a_1 x + a_0$, with
{\scriptsize
\begin{eqnarray*}
a_3 & = &
37909827361040902434390338072754918705969566622865244598340785379492062293493023
\setminus\\&&
07887220632471591953460261515915189503199574055791975955834407879578484212700263
\setminus\\&&
2600401437108457032108586548189769 \\
a_2 & = &
18960350992731066141619447121681062843951822341216980089632110294900985267348927
\setminus\\&&
56700435114431697785479098782721806327279074708206429263751983109351250831853735
\setminus\\&&
1901282000421070182572671506056432 \\
a_1 & = &
69337488142924022910219499907432470174331183248226721112535199929650663260487281
\setminus\\&&
50177351432967251207037416196614255668796808046612641767922273749125366541534440
\setminus\\&&
5882465731376523304907041006464504 \\
a_0 & = &
31678142561939596895646021753607012342277658384169880961095701825776704126204818
\setminus\\&&
48230687778916790603969757571449880417861689471274167016388608712966941178120424
\setminus\\&&
3813332617272038494020178561119564
\end{eqnarray*}}
\indent
The $\rho$-value of $\Jac(C)$ is $8.06$.
\end{ex}

\begin{ex}
Let $K$ be the degree-6 Galois CM field $\Q(\zeta_7)$, and let
$\Phi=\{\phi_1, \phi_2, \phi_3\}$ be the CM
type of $K$ such that $\phi_n(\zeta_7) = e^{2\pi in/7}$.
We used the CM type $(K,\Phi)$ to construct a curve $C$ whose Jacobian has
embedding degree $17$ with
respect to $r = 2^{180} - 7427$.  Since $K$ has class number $1$ and one
equivalence class of primitive CM types, there is a unique isomorphism class of
curves in characteristic zero whose Jacobians are simple and have CM by $K$;
these curves are given by $y^2 = x^7 + a$.  Algorithm
\ref{alg:construct-pi} output the following field size:
{\scriptsize
\begin{eqnarray*}
q & = &
15755841381197715359178780201436879305777694686713746395506787614025008121759749
\setminus\\&&
72634937716254216816917600718698808129260457040637146802812702044068612772692590
\setminus\\&&
77188966205156107806823000096120874915612017184924206843204621759232946263357637
\setminus\\&&
19251697987740263891168971441085531481109276328740299111531260484082698571214310
\setminus\\&&
33499 \quad \mbox{(1077 bits)}
\end{eqnarray*}}
\noindent
The equation of the curve $C$ is $y^2 = x^7 + 10$.  The $\rho$-value
of $\Jac(C)
$ is $17.95$.
\end{ex}

\noindent We conclude with an example of an 8-dimensional abelian
variety found using our algorithms.
We started with a single CM abelian variety $A$ in characteristic
zero and applied our algorithm to different CM-types until we found
a prime $q$ for which the reduction has the given embedding degree.

\begin{ex}
Let $K = \Q(\zeta_{17})$.  We set $r = 1021$ and $k = 10$ and ran
Algorithm \ref
{alg:construct-pi} repeatedly with different CM types for $K$.  Given the
output, we tested the Jacobians of twists of $y^2 = x^{17}+1$ for the specified
number of points.  We found that the curve $y^2 = x^{17} + 30$ has embedding
degree $10$ with respect to $r$ over the field $\F$ of order
\[ q = 6869603508322434614854908535545208978038819437. \]
The CM type was
\[ \Phi = \{\phi_1,\phi_3,\phi_5,\phi_6,\phi_8,\phi_{10},\phi_{13},\phi_{15}\},
\]
where $\phi_n(\zeta_{17}) = e^{2\pi in /17}$.
The $\rho$-value of $\Jac(C)$ is $121.9$.
\end{ex}


\end{document}